\documentclass[11pt]{amsart}

\usepackage[utf8]{inputenc}
\usepackage{mathtools,amsfonts,amsthm,
mathrsfs,amssymb}
\usepackage{textcomp}
\mathtoolsset{showonlyrefs}
\usepackage{bookmark}
\usepackage{microtype}
\usepackage{enumitem}
\usepackage{algorithm}
\usepackage{algorithmic}

\newtheorem{theorem}{Theorem}
\newtheorem{lemma}[theorem]{Lemma}
\newtheorem{proposition}[theorem]{Proposition}
\newtheorem{corollary}[theorem]{Corollary}

\theoremstyle{definition}
\newtheorem*{definition*}{Definition}

\newtheorem*{glll}{The General Lov\'asz Local Lemma}

\newcommand{\EE}{\mathbb{E}}
\newcommand{\EZ}{\mathbb{Z}}
\newcommand{\cD}{\mathcal{D}}
\newcommand{\cE}{\mathcal{E}}
\newcommand{\bI}{\mathbf{I}}
\newcommand{\cI}{\mathcal{I}}
\newcommand{\bX}{\mathbf{X}}
\newcommand{\bZ}{\mathbf{Z}}
\newcommand{\sH}{\mathscr{H}}
\newcommand{\pth}[1]{\left( #1 \right)}

\newcommand{\lam}{\lambda}
\renewcommand{\epsilon}{\varepsilon}
\newcommand{\eps}{\varepsilon}
\renewcommand{\Pr}{\mathbb{P}}

\newcommand{\REALS}{\mathbb{R}}

\newcommand{\val}{\tau}
\DeclareMathOperator{\dom}{dom}

\title{Colouring triangle-free graphs with local list sizes}

\author[E.\ Davies]{Ewan Davies}
\address{Korteweg--De Vries Institute for Mathematics, University of Amsterdam, Netherlands.}
\email{maths@ewandavies.org}
\thanks{(E.\ Davies) The research leading to these results has received funding from the European Research Council under the European Union's Seventh Framework Programme (FP7/2007-2013) / ERC grant agreement \textnumero{} 339109.}
\author[R.\ de Joannis de Verclos]{R\'emi de Joannis de Verclos}
\address{Department of Mathematics, Radboud University Nijmegen, Netherlands.}
\email{r.deverclos@math.ru.nl}
\author[R. J.\ Kang]{Ross J. Kang}
\address{Department of Mathematics, Radboud University Nijmegen, Netherlands.}
\email{ross.kang@gmail.com}
\thanks{(R.\ de Joannis de Verclos, R. J.\ Kang) Supported by a Vidi grant (639.032.614) of the Netherlands Organisation for Scientific Research (NWO)}
\author[F.\ Pirot]{Fran\c{c}ois Pirot}
\address{Department of Mathematics, Radboud University Nijmegen, Netherlands and  LORIA, Universit\'e de Lorraine, Nancy, France.}
\email{francois.pirot@loria.fr}



\begin{document}
\begin{abstract}
We prove two distinct and natural refinements of a recent breakthrough result of Molloy (and a follow-up work of Bernshteyn) on the (list) chromatic number of triangle-free graphs. In both our results, we permit the amount of colour made available to vertices of lower degree to be accordingly lower. One result concerns list colouring and correspondence colouring, while the other concerns fractional colouring.
Our proof of the second illustrates the use of the hard-core model to prove a Johansson-type result, which may be of independent interest.
\end{abstract}

\maketitle

\section{Introduction}

The chromatic number of triangle-free graphs is a classic topic, cf.\ e.g.~\cite{UnDe54,Zyk49}, and has been deeply studied from many perspectives, including algebraic, probabilistic, and algorithmic. It is attractive because of its elegance and its close connection to quantitative Ramsey theory~\cite{AKS81,She83}.

Recently Molloy~\cite{Mol19} obtained a breakthrough by showing that, given $\eps>0$, every triangle-free graph of maximum degree $\Delta$ has chromatic number at most $\lceil(1+\eps)\Delta/\log \Delta\rceil$, provided $\Delta$ is sufficiently large. This achievement improved on the seminal work of Johansson~\cite{Joh96} in two ways, one by lowering the leading asymptotic constant (perhaps even to optimality) and the other by giving a much simpler proof (via entropy compression).

Molloy's result actually guarantees a proper colouring of the graph in the more general situation that every vertex is supplied permissible colour lists of size $\lceil (1+\eps)\Delta/\log \Delta \rceil$. It is natural to ask what happens if fewer colours are supplied to vertices that are not of maximum degree; indeed one might expect the low degree vertices to be easier to colour in a quantifiable way.

The general idea of having ``local'' list sizes is far from new; it can be traced at least back to degree-choosability as introduced in one of the originating papers for list colouring~\cite{ERT80}. Recently Bonamy, Kelly, Nelson, and Postle~\cite{BKNP18+} initiated a modern and rather general treatment of this idea, including with respect to triangle-free graphs. (A conjecture of King~\cite{Kin09thesis} and related work are in the same vein.)
We show the following result.

\begin{theorem}\label{thm:localmolloy}
Fix $\eps > 0$, let $\Delta$ be sufficiently large, and let $\delta=(192\log \Delta)^{2/\eps}$. Let $G$ be a triangle-free graph of maximum degree $\Delta$ and $L : V(G)\to 2^{\EZ^+}$ be a list assignment of $G$ such that for all $v\in V(G)$,
\[
|L(v)| \ge (1+\eps)\max\Bigg\{\frac{\deg(v)}{\log \deg(v)},\,\frac{\delta}{\log \delta}\Bigg\}\,,
\] 
Then there exists a proper colouring $c: V(G) \to \EZ^+$ of $G$ such that $c(v) \in L(v)$ for all $v\in V(G)$.
\end{theorem}

\noindent This of course implies Molloy's theorem, and can be considered a local strengthening. When the graph $G$ in Theorem~\ref{thm:localmolloy} is of minimum degree $\delta$, the list size condition is local in the sense that the lower bound on $|L(v)|$ reduces to a function of $\deg(v)$ and no other parameter of $G$. 
Theorem~\ref{thm:localmolloy} (or rather the stronger Theorem~\ref{thm:localmolloyDP} below) improves upon~\cite[Thm.~1.12]{BKNP18+}, by having an asymptotic leading constant of $1$ rather than $4\log 2$, at the expense of requiring a larger minimum list size. Our proof relies heavily on the work of Bernshteyn~\cite{Ber19}, who gave a further simplified proof for a stronger version of Molloy's theorem. For Theorem~\ref{thm:localmolloy}, it has sufficed to prove a local version of the so-called ``finishing blow'' (see Lemma~\ref{lem:finishingblow} below) and to notice that there is more than enough slack in Bernshteyn's (and indeed Molloy's) argument to satisfy the new blow's hypothesis.

We also provide a local version of Molloy's theorem for a relaxed, fractional form of colouring. Writing $\cI(G)$ for the set of independent sets of $G$, and $\mu$ for the standard Lesbegue measure on $\REALS$, a \emph{fractional colouring} of a graph $G$ is an assignment $w(I)$ for $I\in \cI(G)$ of pairwise disjoint measurable subsets of $\REALS$ to independent sets such that $\sum_{I\in\cI(G), I \ni v}\mu(w(I))\ge 1$  for all $v\in V(G)$. 
Such a colouring naturally induces an assignment of measurable subsets to the vertices of $G$, namely $w(v) =\bigcup_{I\in\cI(G), I \ni v} w(I)$ for each $v\in V(G)$, such that $w(u)$ and $w(v)$ are disjoint whenever $uv\in E(G)$. 
The total weight of the fractional colouring is $\hat w(G)=\sum_{I\in\cI(G)}\mu(w(I))$.

\begin{theorem}
\label{thm:fracmolloy}
For all $\eps>0$ there exists $\delta>0$ such that every triangle-free graph $G$ admits a fractional colouring $w$ such that for every $v\in V(G)$
\[
w(v) \subseteq \left[0,(1+\eps)\max\Big\{\frac{\deg(v)}{\log \deg(v)},\,\frac{\delta}{\log\delta}\Big\}\right).
\]
\end{theorem}

\noindent
Again, when $G$ is of minimum degree $\delta$ our condition on $w(v)$ reduces to a function of $\deg(v)$ alone, yielding a local condition. 
Clearly Theorem~\ref{thm:fracmolloy} is not implied by Molloy's theorem nor is the converse true, but both results imply that the fractional chromatic number of a triangle-free graph of maximum degree $\Delta$ is at most $(1+o(1))\Delta/\log\Delta$.
We believe that the main interest in Theorem~\ref{thm:fracmolloy} will be in its derivation. We give a short and completely self-contained proof by analysing a probability distribution on independent sets known as the \emph{hard-core model} in triangle-free graphs (Lemma~\ref{lem:hcmbound}), and demonstrating that to obtain the desired result it suffices to feed this distribution as input to a greedy fractional colouring algorithm (Lemma~\ref{lem:chifalg}). Since it makes no use of the Lov\'asz Local Lemma, the proof is unlike any other derivation of a Johansson-type colouring result (regardless of local list sizes). This may be of independent interest.

The asymptotic leading constant of $1$ in the conditions of both Theorems~\ref{thm:localmolloy} and~\ref{thm:fracmolloy} cannot be improved below $1/2$ due to random regular graphs~\cite{FrLu92}. 
In fact, as a corollary of either result we match asymptotically the upper bound of Shearer~\cite{She83} for off-diagonal Ramsey numbers. So any improvement below $1$, or even to $1$ precisely (i.e.~removal of the $\eps$ term), would be a significant advance.
To give more detail, Shearer proved\footnote{In fact, Shearer proved a strengthening of this bound with $\Delta$ replaced by the average degree of $G$.} that as $\Delta\to\infty$ any triangle-free graph on $n$ vertices of maximum degree $\Delta$ contains an independent set of size at least 
\begin{equation}\label{eq:shearer}
(1+o(1))\frac{n\log\Delta}{\Delta}.
\end{equation} 
It is easy to show that any graph contains an independent set of size at least $n/\chi$ if it permits any of a fractional colouring with total weight $\chi$, a proper colouring with $\chi$ colours, or an $L$-colouring whenever $|L(v)|\ge\chi$ for all vertices $v$, and hence the leading constant in the bound of Molloy, and in Theorems~\ref{thm:localmolloy} and~\ref{thm:fracmolloy} cannot be improved without improving this `Shearer bound' on the independence number of triangle-free graphs. 
Our analysis of the hard-core model on triangle-free graphs has its roots in~\cite{davies2018average}, where the first author, Jenssen, Perkins, and Roberts showed that for $G$ as above,~\eqref{eq:shearer} is a lower bound on the expected size of an independent set from the hard-core model (when a parameter known as the \emph{fugacity} is not too small). 
Most intriguingly, they proved that their result is asymptotically tight by appealing to the random regular graph, whereas Theorem~\ref{thm:fracmolloy} is not known to be tight: there is a factor two gap between the fractional chromatic number of the random regular graph and the bound one gets via Theorem~\ref{thm:fracmolloy} or Molloy's result. 
We define the hard-core model and the fugacity parameter in Section~\ref{sec:hardcore}.

Allow us to make some further remarks related to the maxima that occur in the list and weight conditions of Theorems~\ref{thm:localmolloy} and~\ref{thm:fracmolloy}, which at first sight seem artificial and unnecessary. 
In Theorem~\ref{thm:fracmolloy} the parameter $\delta$ is a function of $\eps$ alone but in Theorem~\ref{thm:localmolloy} we require $\delta$ to grow with $\Delta$.
So for large enough $\Delta$ the value of $\delta$ in Theorem~\ref{thm:fracmolloy} is strictly smaller\footnote{A crude estimation shows that (for $\eps$ smaller than some absolute constant) our proof of Theorem~\ref{thm:fracmolloy} permits $\delta=(3/\eps)^{3/\eps}$, so for $\Delta \ge \exp(1/\eps^2)$ this would occur.} than the value in Theorem~\ref{thm:localmolloy}, and since the list chromatic number can be much larger than the fractional chromatic number (even for bipartite graphs) neither of Theorems~\ref{thm:localmolloy} and~\ref{thm:fracmolloy} implies the other.
In Section~\ref{sec:necessary}, we show that these are truly distinct results in that, unlike in Theorem~\ref{thm:fracmolloy}, some non-trivial (albeit very slight) dependence between minimum list size and maximum degree is necessary in Theorem~\ref{thm:localmolloy}. 
Last observe that, if we were able to improve either result by lowering $\delta$ to a quantity independent of $\eps$, then it would constitute a significant improvement over Shearer's bound.

We are hopeful that some of the techniques we used in this paper might also be applicable to other natural colouring problems in triangle-free graphs, such as bounding the (list) chromatic number in terms of the number of vertices, cf.~\cite[Conjs.~4.3 and~6.1]{CJKP18+}, but leave this for further investigation.

\subsection{Structure of the paper}

In Section~\ref{sec:fracgreedy}, we prove a greedy fractional colouring lemma (Lemma~\ref{lem:chifalg}).
We give a local analysis of the hard-core model in triangle-free graphs in Section~\ref{sec:hardcore}, culminating in Lemma~\ref{lem:hcmbound}. As a demonstration of its further applicability, we also use Lemma~\ref{lem:hcmbound} to give a good bound on semi-bipartite induced density in triangle-free graphs (Theorem~\ref{thm:semibipartite}), a concept related to a recent conjecture of Esperet, Thomass\'e and the third author~\cite{esperet2018separation}.
We prove Theorem~\ref{thm:fracmolloy} in Section~\ref{sec:fracmolloy}.
In Section~\ref{sec:finishingblow}, we review the definition of correspondence colouring and prove for it a local version of the ``finishing blow'' (Lemma~\ref{lem:finishingblow}). In Section~\ref{sec:localmolloy} we sketch how Bernshteyn's argument can then be adapted to prove Theorem~\ref{thm:localmolloy}.
In Section~\ref{sec:necessary}, we present a simple construction (Proposition~\ref{prop:necessary}) to show that even some bipartite graphs cannot satisfy the conclusions of Theorem~\ref{thm:localmolloy} without a suitable lower bound on $\delta$.

\subsection{Notation and preliminaries}
\label{sub:def}

For a graph $G$ and vertex $v\in V(G)$, we write $N_G(v)$ for the set of neighbours of $v$ in a graph $G$, and $\deg_G(v)=|N_G(v)|$ for the degree of a vertex, where we omit the subscript $G$ if it is clear from context. 
For $i\ge0$ we write $N_G^i(v)$ for the set of vertices in $G$ at distance exactly $i$ from $v$, so that e.g.\ $N_G^0(v)=\{v\}$, and $N_G^1(v)=N_G(v)$.
We have already indicated above that $\cI(G)$ denotes the set of independent sets of $G$. Note that $\varnothing\in\cI(G)$ for all $G$.

The function $W$ is the inverse of $z\mapsto ze^z$, also known as the \emph{Lambert $W$-function},  which satisfies $W(x)=\log x - \log\log x + o(1)$ as $x\to\infty$.

We will have use for the following probabilistic tool, see~\cite{AS16book}.

\begin{glll}
Consider a set $\cE=\{A_1,\dots,A_n\}$ of (bad) events such that
each $A_i$ is mutually independent of $\cE-(\cD_i\cup A_i)$, for some $\cD_i \subseteq \cE$.
If we have reals $x_1,\dots,x_n \in[0,1)$ such that for each $i$
\[
\Pr(A_i) \le x_i \prod_{A_j\in \cD_i} (1-x_j),
\]
then the probability that no event in $\cE$ occurs is at least $\prod_{i=1}^n \limits  (1-x_i)>0$.
\end{glll}

\section{A fractional colouring algorithm}
\label{sec:fracgreedy}

The following result for local fractional colouring is slightly stronger than what we require in the proof of Theorem~\ref{thm:fracmolloy}, but the proof is no different from that needed for the weaker statement.

\begin{lemma}\label{lem:chifalg}
Fix a positive integer $r$.
Let $G$ be a graph and suppose that for every vertex $v\in V(G)$ we have a list $(\alpha_j(v))_{j=0}^r$ of $r+1$ real numbers.
Suppose that for all induced subgraphs $H$ of $G$, there is a probability distribution on $\cI(H)$ such that, writing $\bI_H$ for the random independent set from this distribution, for each $v\in V(H)$ we have the bound
\[\sum_{j=0}^r \alpha_j(v)\EE\big|N_H^j(v)\cap \bI_H\big| \geq 1.\]
Then there exists a fractional colouring of $G$ such that every $v\in V(G)$ is coloured with a subset of the interval
$\big[0,\sum_{j=0}^r \alpha_j(v)\big|N_G^j(v)\big|\big)$.
\end{lemma}

\begin{proof}
We present a refinement of an algorithm given in the book of Molloy and Reed~\cite{MoRe02}, and show that under the assumptions of the lemma, it returns the desired fractional colouring.
The idea of the algorithm is to greedily add weight to independent sets according to the probability distribution induced on all not yet fully coloured vertices.
For brevity, we write $\gamma(v) = \sum_{j=0}^r \alpha_j(v)\big|N_G^j(v)\big|$.

We build a fractional colouring $w$ in several iterations, and we write $\hat w(I)$ for $\mu(w(I))$ so that $\hat w(I)$ is a non-negative integer representing the measure $w$ assigns to $I$. 
Through the iterations, $w$ is a partial fractional colouring in the sense of not yet having satisfied the condition that $\sum_{I\in\cI(G), I \ni v}\hat w(I)\ge 1$  for all $v\in V(G)$. 
We extend our notational conventions for $w$ to $\hat w$, so that $\hat{w}(G) = \sum_{I\in \cI(G)} \hat{w}(I)$ is the total measure used by the current partial colouring, and $\hat{w}(v) = \sum_{I\in\cI(G), I \ni v} \hat{w}(I)$ for any $v\in V(G)$ is the total measure given to a vertex $v$ by the current partial colouring.

\begin{algorithm}\caption{The greedy fractional colouring algorithm}\label{alg-greedy}
    \begin{algorithmic}[0]
        \FOR{$I\in \cI(G)$}
            \STATE$\hat{w}(I)\gets0$
        \ENDFOR
        \STATE $H\gets G$
        \WHILE{$|V(H)|>0$}
            \STATE $\displaystyle \val \gets \min\left\{ \min_{v\in V(H)} \frac{1-\hat{w}(v)}{\Pr(v\in \bI_H)}, \min_{v\in V(H)} \gamma(v)-\hat{w}(G)\right\}$
            \FOR{$I\in\cI(H)$}
                \STATE $\hat{w}(I)\gets\hat{w}(I)+\Pr(\bI_{H}=I) \val$
            \ENDFOR
            \STATE $H \gets H - \{v \in V(H) : \hat{w}(v)=1\}$
        \ENDWHILE
    \end{algorithmic}
\end{algorithm}

%

We next show that Algorithm~\ref{alg-greedy} certifies the desired fractional colouring.
For the analysis, it is convenient to index the iterations: for $i=0,1,\dots$, let $H_i$, $\hat{w}_i(I)$, $\hat{w}_i(v)$, $\hat{w}_i(G)$, $\val_i$ denote the corresponding $H$, $\hat{w}(I)$, $\hat{w}(v)$, $\hat{w}(G)$, $\val$ in the $i$th iteration prior to updating the sequence.
Note then that $H_0 \supseteq H_1 \supseteq H_t \supseteq \cdots$. We also have $\hat{w}_{i+1}(v) = \sum_{k=0}^i \Pr(v\in \bI_{H_k})\val_k$ for any $v\in V(H_i)$ and $\hat{w}_{i+1}(G) = \sum_{k=0}^i \val_k$.

Let us first describe the precise fractional colouring (rather than its sequence of measures) that is constructed during Algorithm~\ref{alg-greedy}.
During the update from $\hat{w}_i$ to $\hat{w}_{i+1}$, in actuality we do the following.
Divide the interval $[\hat{w}_i(G),\hat{w}_i(G)+\val_i)$ into a sequence $(B_I)_{I\in \cI(G)}$ of consecutive right half-open intervals such that $B_I$ has length $\Pr(\bI_{H_i}=I)\val_i$. We then let $w_{i+1}(I) = w_i(I)\cup B_I$ for each $I\in \cI(G)$. 
Note that $\mu(w_i(I)) = \hat{w}_i(I)$ for all $I\in \cI(G)$ and $i$. Moreover, by induction, $w_i(G) \subseteq [0,\hat{w}_i(G))$ for all $i$.

By the choice of $\val_i$, if there is some $v\in V(H_i)$ (i.e.~with $\hat{w}_i(v) < 1$), then $\hat{w}_{i+1}(G) \le \gamma(v)$ and so $w_{i+1}(G) \subseteq [0,\gamma(v))$. So we only need to show that Algorithm~\ref{alg-greedy} terminates. To do so, it suffices to show that $|V(H_{i+1})| < |V(H_i)|$ for all $i$.

If
\[\val_i = \min_{v\in V(H_i)} \frac{1-\hat{w}_i(v)}{\Pr(v\in \bI_{H_i})}\,,
\]
then there must be some $v\in V(H_i)$ such that $\hat{w}_i(v) < 1$ and $\hat{w}_{i+1}(v)=1$, so $|V(H_{i+1})| < |V(H_i)|$ and we are done. We may therefore assume that there is some $v \in V(H_i)$ such that $\val_i = \gamma(v)-\hat{w}_i(G)$, and so $\hat{w}_{i+1}(G)=\gamma(v)$. 

For any $k\in\{0,\dots,i\}$, we know that
\[ \sum_{j=0}^r \alpha_j(v)\EE\big|N_{H_k}^j(v)\cap \bI_{H_k}\big| \geq 1\,,\] and so 
\[ \sum_{j=0}^r \alpha_j(v)\sum_{u\in N^j_{H_k}(v)} \Pr(u\in \bI_{H_k})\val_k \geq \val_k.\]
By summing this last inequality over all such $k$, we obtain
\begin{align}\label{eq:gammaveqs}
\sum_{j=0}^r \alpha_j(v) \big|N^j_G(v)\big| \ge \sum_{j=0}^r \alpha_j(v) \sum_{u\in N^j_G(v)} \hat{w}_{i+1}(u) &\geq \hat{w}_{i+1}(G) = \gamma(v),
\end{align}
and $\gamma(v)$ is defined to be the left-hand side of this chain of inequalities, so we have equality throughout. 
We also note that 
\begin{align*}
	\sum_{j=0}^r \alpha_j(v) \sum_{u\in N^j_G(v)} \hat{w}_{i+1}(u) 
	&= \alpha_0(v)\hat{w}_{i+1}(v) + \sum_{j=1}^r \alpha_j(v) \sum_{u\in N^j_G(v)} \hat{w}_{i+1}(u)
\\&\le \alpha_0(v)\hat{w}_{i+1}(v) + \sum_{j=1}^r \alpha_j(v) \big|N^j_G(v)\big|
\end{align*}
because for any $u$ appearing in the sum on the right-hand side of the first line we have $0\le \hat{w}_{i+1}(u)\le 1$. Indeed, the choice of $\val_i$ in the algorithm ensures that the weights $\hat{w}_{i+1}(u)$ never exceed $1$. 
We then have from~\eqref{eq:gammaveqs} that
\begin{align*}
\alpha_0(v) \hat{w}_{i+1}(v) + \sum_{j=1}^r \alpha_j(v) |N^j_{G}(v)| &\geq \gamma(v) = \alpha_0(v) + \sum_{j=1}^r \alpha_j(v) |N_{G}^j(v)|,
\end{align*}
giving $\alpha_0(v) \hat{w}_{i+1}(v) \geq \alpha_0(v)$ and hence that $\hat{w}_{i+1}(v)=1$. 
This means $|V(H_{i+1})| < |V(H_i)|$, as required for a proof of termination.
\end{proof}

\section{A local analysis of the hard-core model}
\label{sec:hardcore}

Given a graph $G$, and a parameter $\lam>0$, the \emph{hard-core model on $G$ at fugacity $\lam$} is a probability distribution on the independent sets $\cI(G)$ (including the empty set) of $G$, where each $I\in\cI(G)$ occurs with probability proportional to $\lam^{|I|}$. Writing $\bI$ for the random independent set, we have
\[
\Pr(\bI=I) = \frac{\lam^{|I|}}{Z_G(\lam)}\,,
\]
where the normalising term in the denominator is the \emph{partition function} (or independence polynomial) $Z_G(\lam)=\sum_{I\in\cI(G)} \lam^{|I|}$. 

Given a choice of $I\in\cI(G)$, we say that a vertex $u\in V(G)$ is \emph{uncovered} if $N(u)\cap I=\varnothing$, and that $u$ is \emph{occupied} if $u\in I$. 
Note that $u$ can be occupied only if it is uncovered. 

For the rest of this section we assume that $G$ is triangle-free. We note the following useful facts (which appear verbatim in~\cite{davies2017independent,davies2018average}). 

\begin{description}
\item[Fact 1] $\Pr(v\in\bI | v \text{ uncovered}) = \frac{\lam}{1+\lam}$.
\item[Fact 2] $\Pr(v \text{ uncovered} | v \text{ has $j$ uncovered neighbours}) = (1+\lam)^{-j}$.
\end{description}

Fact 1 holds because, for each realisation $J$ of $\bI\setminus\{v\}$ such that $J\cap N(v)=\varnothing$ (i.e.\ $v$ is uncovered), there are two possible realisations of $\bI$, namely $J$ and $J \cup\{v\}$.
Now, $\bI$ takes these values with probabilities proportional to $\lam^{|J|}$ and $\lam^{1+|J|}$ respectively, so for such $J$ we have 
\[
\Pr(v\in\bI ~|~ \bI\setminus\{v\} = J)=\frac{\lam^{1+|J|}}{\lam^{|J|}+\lam^{1+|J|}}\,,
\]
and the fact follows.

Fact 2 holds because, for each realisation $J$ of $\bI \setminus N(v)$ such that $|N(v) \setminus N(J)|=j$, every possible subset of $N(v) \setminus N(J)$ (the uncovered neighbours of $v$) extends $J$ into a valid realisation of $\bI$. Only the empty set extends $J$ into a realisation of $\bI$ where $v$ is uncovered, so we have
\[ \Pr(v \mbox{ uncovered} ~|~ \bI \setminus N(v) = J) = \frac{\lambda^{|J|}}{\sum_{X \subseteq N(v)\setminus N(J)}\limits \lambda^{|X|+|J|}} = \pth{1+\lambda}^{-j}.\]

We apply these facts to give a lower bound on a linear combination of the probability that $v$ is occupied and the expected number of occupied neighbours of $v$. 
This is a slight modification of the arguments of~\cite{davies2017independent,davies2018average}, but here we focus on individual vertices, rather than averaging over a uniformly random choice of vertex. 

\begin{lemma}\label{lem:hcmbound}
Let $G$ be a triangle-free graph and let $(\alpha_v)_{v\in V(G)}$ and $(\beta_v)_{v\in V(G)}$ be sequences of positive real numbers. Write $\bI$ for a random independent set drawn from the hard-core model on $G$ at fugacity $\lam>0$. Then for every $v\in V(G)$, we have
\[ 
\alpha_v \Pr(v\in\bI) + \beta_v\EE|N(v)\cap \bI| \ge \frac{\beta_v\lam  \left(\log(\alpha_v/\beta_v) + \log\log(1+\lam)+1\right)}{(1+\lam) \log(1+\lam)}
\]
\end{lemma}

\begin{proof}
Fix a vertex $v\in V(G)$ and let $\bZ$ be the number of uncovered neighbours of $v$ given the random independent set $\bI$. 
By Fact 1, conditioning on the number of uncovered neighbours of $v$, and by Fact 2, we have 
\begin{align}
\Pr(v\in\bI) 
   &= \frac{\lam}{1+\lam}\Pr(v\text{ uncovered})
\\ &= \frac{\lam}{1+\lam}\sum_{j\ge 0}(1+\lam)^{-j}\cdot\Pr(v \text{ has $j$ uncovered neighbours})
\\ &= \frac{\lam}{1+\lam}\EE\big[(1+\lam)^{-\bZ}\big] \ge \frac{\lam}{1+\lam}(1+\lam)^{-\EE \bZ}\,,
\end{align}
where for the final inequality we used Jensen's inequality. 
Similarly, each of the $\bZ$ uncovered neighbours of $v$ is occupied with probability $\lam/(1+\lam)$ independently of the others (since $G$ is triangle-free), and a covered neighbour of $v$ is occupied with probability zero. Hence
\[
\EE|N(v)\cap\bI| = \frac{\lam}{1+\lam}\EE\bZ.
\]
Then for any vertex $v\in V(G)$ and positive reals $\alpha_v$ and $\beta_v$ we have
\begin{align}
	\alpha_v \Pr(v\in\bI) + \beta_v\EE|N(v)\cap \bI| 
	  &\ge \frac{\lam}{1+\lam}\Big(\alpha_v(1+\lam)^{-\EE \bZ} + \beta_v\EE\bZ\Big)\,,
\end{align}
and since $\EE\bZ$ is some (non-negative) real number we also have
\begin{align}
	\alpha_v \Pr(v\in\bI) + \beta_v\EE|N(v)\cap \bI| 
	  &\ge \frac{\lam}{1+\lam}\min_{z\in \REALS}\Big\{\alpha_v(1+\lam)^{-z} + \beta_v z\Big\}\,.
\end{align}
Let $g(z) := \alpha_v(1+\lam)^{-z} + \beta_v z$. 
When $\alpha_v,\lam>0$ the function $g$ is strictly convex (because its second derivative is positive), and hence has a unique stationary point at $z=z^*$, say, which gives its minimum. 
We compute that 
\[ g'(z^*) = \beta_v -\alpha_v\frac{\log(1+\lam)}{(1+\lam)^{z^*}} = 0 \Longleftrightarrow z^* = \frac{\log(\alpha_v/\beta_v) + \log\log(1+\lam)}{\log(1+\lam)}\,, \]
showing that for every vertex $v\in V(G)$ we have 
\begin{align}
\alpha_v \Pr(v\in\bI) + \beta_v\EE|N(v)\cap \bI| &\ge g(z^*) 
\\&= \frac{\beta_v\lam  \left(\log(\alpha_v/\beta_v) + \log\log(1+\lam)+1\right)}{(1+\lam) \log(1+\lam)}\,.\qedhere
\end{align}
\end{proof}

We next give a result related to a recent conjecture of Esperet, Thomass\'e and the third author~\cite[Conj.~1.5]{esperet2018separation}.
A \emph{semi-bipartite induced subgraph} of a graph $G$ is a subgraph $H$ of $G$ consisting of all edges between two disjoint subsets $A,B\subset V(G)$ such that $A$ is independent. 
This definition means that average degree of such a semi-bipartite induced subgraph $H$ is $\frac{2}{|A|+|B|}e_G(A,B)$, where $e_G(A,B)$ represents the number of edges of $G$ with one endpoint in $A$ and one endpoint in $B$.
Our local analysis of the hard-core model in triangle-free graphs yields a semi-bipartite induced subgraph of high average degree, measured by a property of $G$ that incorporates local degree information: the geometric mean of the degree sequence. 
We improve upon~\cite[Thm.~3.5]{esperet2018separation} by replacing minimum degree with the geometric mean of the degrees, and increasing the leading constant. 

\begin{theorem}\label{thm:semibipartite}
A triangle-free graph $G$ on $n$ vertices contains a semi-bipartite induced subgraph of average degree at least $(2+o(1)) \frac{1}{n}\sum_{v\in V(G)}\log \deg(v)$.
\end{theorem}

In the statement of the theorem and in the proof below, the $o(1)$ term tends to zero as the geometric mean of the degree sequence of $G$ tends to infinity.

\begin{proof}[Proof of Theorem~\ref{thm:semibipartite}]
We find a semi-bipartite induced subgraph of $G$ where one of the parts is a random independent set $\bI$ from the hard-core model, and the other is $V(G)\setminus \bI$. 
The number of edges between the parts is therefore $e_G(\bI,V(G)\setminus\bI) = \sum_{v\in \bI}\deg(v)$, which is a random variable we denote $\bX$. 
We write $\EE\bX$ in two different ways: 
\begin{align}
\EE\bX
  &= \sum_{v\in V(G)}\deg(v)\Pr(v\in \bI) = \sum_{v\in V(G)}\EE|N(v)\cap\bI|.
\end{align}
The first version follows from linearity of expectation, and for the second we note that $\EE|N(v)\cap\bI|=\sum_{u\in N(v)}\Pr(u\in\bI)$ and hence $\Pr(u\in\bI)$ appears $\deg(u)$ times in the sum as required. 
For brevity, we write $\sum_v$ for a sum over $v\in V(G)$ in the rest of the proof. 
Then for any $\alpha, \beta>0$ we have 
\[
(\alpha+\beta)\EE\bX=\sum_v\Big(\alpha \deg(v)\Pr(v\in\bI)+\beta\EE|N(v)\cap\bI|\Big)\,,
\]
hence by Lemma~\ref{lem:hcmbound}, 
\[
\EE\bX \ge \frac{n\lam  \left(\frac{1}{n}\sum_v\log \deg(v)+\log(\alpha/\beta) + \log\log(1+\lam)+1\right)}{(1+\alpha/\beta)(1+\lam) \log(1+\lam)}.
\]
Choosing e.g.\ $\alpha/\beta = \lam = n/\sum_v \log \deg(v)$, we observe that 
\[
\EE\bX\ge (1+o(1))\sum_v\log \deg(v).
\]
To complete the proof, note that the bound on $\EE\mathbf X$ means that there is at least one independent set $I$ with at least $(1+o(1))\sum_v\log \deg(v)$ edges from $I$ to its complement. 
This immediately means that the average degree of the semi-bipartite subgraph with parts $I$ and $V(G)\setminus I$ is at least $(2+o(1))\frac{1}{n}\sum_v\log \deg(v)$. 
\end{proof}

We remark that the methods of~\cite{davies2018average} deal with the quantities $\Pr(v\in\bI)$ and $\EE|N(v)\cap\bI|$ in a slightly more sophisticated manner that avoids the seemingly arbitrary parameter $\alpha/\beta$ in the above proof. 
Since we have Lemma~\ref{lem:hcmbound} for other purposes in this paper, it is expedient to use it here.

\section{Local fractional colouring}
\label{sec:fracmolloy}

\begin{proof}[Proof of Theorem~\ref{thm:fracmolloy}]
The method is to combine Lemmas~\ref{lem:chifalg} and~\ref{lem:hcmbound} by carefully choosing $(\alpha_v)_{v\in V(G)}$ and $(\beta_v)_{v\in V(G)}$. 
For every $v\in V(G)$, we want to minimise $\alpha_v+\beta_v\deg(v)$ subject to the condition
\begin{align}\label{eq:ab1}
\frac{\beta_v\lam  \left(\log(\alpha_v/\beta_v) + \log\log(1+\lam)+1\right)}{(1+\lam) \log(1+\lam)} = 1.
\end{align}
For then the hypothesis of Lemma~\ref{lem:chifalg} (with $\alpha_0(v)=\alpha_v$ and $\alpha_1(v)=\beta_v$ for all $v\in V(G)$) follows from the conclusion of Lemma~\ref{lem:hcmbound}. 
Given the assumptions on $G$, we can apply Lemma~\ref{lem:hcmbound} to any induced subgraph $H$ of $G$ since such $H$ are also triangle-free and the local parameters $\alpha_v$ and $\beta_v$ are invariant under taking induced subgraphs.

Note that~\eqref{eq:ab1} is equivalent to
\[
\alpha_v = \frac{\beta_v  (1+\lam)^{\frac{1+\lam}{\beta_v  \lam }}}{e\log(1+\lam)}\,,
\]
so that $\alpha_v+\beta_v\deg(v)$ is a convex function of $\beta_v$ with a minimum at 
\[
\beta_v = \frac{1+\lam}{\lam} \cdot \frac{\log (1+\lam)}{1 +  W\big(\deg(v)  \log (1+\lam)\big)}\,,
\]
giving 
\[
\alpha_v+\beta_v\deg(v) = \frac{1+\lam}{\lam} \cdot e^{W(\deg(v)  \log (1+\lam))}.
\]
For any fixed $\lam$ this is an increasing function of $\deg(v)$. We take $\lam=\eps/2$, and we are done by Lemma~\ref{lem:chifalg} if we can show that there exists $\delta>0$ such that for all $\deg(v)\ge \delta$ we have
\begin{align}\label{eq:done}
(2/\eps+1) \cdot e^{W(\deg(v) \log (1+\eps/2))} \le (1+\eps)\frac{\deg(v)}{\log \deg(v)}\,.
\end{align}
Let us first assume that $\deg(v)$ is at least some large enough multiple of $1/\eps$ so that
\[
e^{W(\deg(v) \log (1+\eps/2))} \le \frac{(1+\eps/2) \deg(v) \log (1+\eps/2)}{\log(\deg(v) \log (1+\eps/2))},
\]
where we used the fact that $W(x)=\log x - \log\log x + o(1)$ as $x\to\infty$.
Then by~\eqref{eq:done}, it suffices to have that
\[
(2/\eps+1)(1+\eps/2)\log (1+\eps/2) \cdot \log \deg(v) \le (1+\eps)\log(\deg(v) \log (1+\eps/2)).
\]
This last inequality holds for $\deg(v)$ large enough (as a function of $\eps$) provided
\[
(2/\eps+1)(1+\eps/2)\log (1+\eps/2) < 1+\eps.
\]
This is easily checked to hold true for small enough $\eps$, namely $\eps \le 4$.
\end{proof}

\section{A list colouring lemma}
\label{sec:finishingblow}

Just as in~\cite{Ber19}, we will establish Theorem~\ref{thm:localmolloy} for a generalised form of list colouring called {\em correspondence colouring} (or {\em DP-colouring}). We here state the definition given in~\cite{Ber19}.

		Given a graph $G$, a \emph{cover} of $G$ is a pair $\sH = (L, H)$, consisting of a graph $H$ and a function $L \colon V(G) \to 2^{V(H)}$, satisfying the following requirements:
		\begin{enumerate}
			\item the sets $\{L(u) \,:\,u \in V(G)\}$ form a partition of $V(H)$;
			\item for every $u \in V(G)$, the graph $H[L(u)]$ is complete;
			\item if $E_H(L(u), L(v)) \neq \varnothing$, then either $u = v$ or $uv \in E(G)$;
			\item \label{item:matching} if $uv \in E(G)$, then $E_H(L(u), L(v))$ is a matching (possibly empty).
		\end{enumerate}
		An \emph{$\sH$-colouring} of $G$ is an independent set in $H$ of size $|V(G)|$.

A reader who prefers not to concern herself with this generalised notion may merely read $L$ as an ordinary list assignment and $V(H)$ as the disjoint union of all lists. For usual list colouring, there is an edge in $H$ between equal colours of two lists if and only if there is an edge between their corresponding vertices in $G$.

To state and prove our local version of the finishing blow, we will need some further notation.
Define $H^*$ to be the spanning subgraph of $H$ such that an edge $c_1c_2 \in E(H)$ belongs to $E(H^*)$ if and only if $c_1$ and $c_2$ are in different parts of the partition $\{L(u)\,:\, u \in V(G)\}$. We write $\deg^*_{\sH}(c)$ instead of $\deg_{H^*}(c)$.

\begin{lemma}\label{lem:finishingblow}
Let $\sH = (L,H)$ be a cover of a graph $G$. Suppose there is a function $\ell : V(G) \to \EZ_{\ge 3}$, such that, for all $u\in V(G)$, $|L(u)| \ge \ell(u)$ and $\deg^*_\sH(c) \le \tfrac18 \min_{v\in N_G(u)} \ell(v)$ for all $c \in L(u)$. Then $G$ is $\sH$-colourable. 
\end{lemma}

For clarity, we separately state the corollary this lemma has for conventional list colouring.

\begin{corollary}\label{cor:finishingblow}
Let $L : V(G)\to 2^{\EZ^+}$ be a list assignment of a graph $G$. Suppose there is a function $\ell : V(G) \to \EZ_{\ge 3}$ such that, for all $u\in V(G)$, $|L(u)| \ge \ell(u)$ and the number of neighbours $v\in N_G(u)$ for which $L(v)\ni c$ is at most $\tfrac18 \min_{v\in N_G(u)} \ell(v)$ for all $c \in L(u)$.
Then there exists a proper colouring $c: V(G) \to \EZ^+$ of $G$ such that $c(u) \in L(u)$ for all $u\in V(G)$. 
\end{corollary}

\begin{proof}[Proof of Lemma~\ref{lem:finishingblow}]
Remove, if needed, some vertices from $H$ to ensure that $|L(u)|=\ell(u)$ for all $u\in V(G)$.
Let $\bI$ be a random subset of $V(H)$ obtained by choosing, independently and uniformly, one vertex from each list $L(u)$.
For $c_1c_2\in E(H^*)$, let $B_{c_1c_2}$ denote the event that both $c_1$ and $c_2$ are chosen in $\bI$. So, if none of the events $B_{c_1c_2}$ occur, then $\bI$ is an independent set and hence an $\sH$-colouring. Let $u_i$ be the vertex of $G$ such that $c_i \in L(u_i)$, for $i\in \{1,2\}$. By definition, $\Pr(B_{c_1c_2})=(\ell(u_1)\ell(u_2))^{-1}$. Define
\[
\Gamma(c_1c_2) = \{c_1'c_2' \in E(H^*) \, : \, c_1'\in L(u_1)\text{ or }c_2'\in L(u_2) \}.
\]
Note that $B_{c_1c_2}$ is mutually independent of the events $B_{c_1'c_2'}$ with $c_1'c_2' \notin \Gamma(c_1c_2)$.
All that remains is to define weights $x_{c_1c_2} \in [0,1)$ to satisfy the hypothesis of the General Lov\'asz Local Lemma.
In particular, we need that
\begin{align*}
(\ell(u_1)\ell(u_2))^{-1}=\Pr(B_{c_1c_2}) \le x_{c_1c_2} \prod_{c_1'c_2'\in \Gamma(c_1c_2)} (1-x_{c_1'c_2'}).
\end{align*}
Since $\exp(-1.4x) \le 1 - x$ if $0 \le x < 0.5$, it suffices to find weights $x_{c_1c_2} \in [0,0.5)$ satisfying
\begin{align}\label{eqn:glllhyp1}
(\ell(u_1)\ell(u_2))^{-1} \le x_{c_1c_2} \exp\left(-1.4\sum_{c_1'c_2'\in \Gamma(c_1c_2)}x_{c_1'c_2'}\right).
\end{align}
If we choose weights of the form $x_{c_1c_2} = k(\ell(u_1)\ell(u_2))^{-1}$ for some constant $k>0$, then~\eqref{eqn:glllhyp1} becomes
\begin{align*}
\log k \ge 1.4 k\sum_{c_1'c_2'\in \Gamma(c_1c_2)}(\ell(u_1')\ell(u_2'))^{-1}
\end{align*}
(where $u'_i$ is such that $c'_i \in L(u'_i)$, for $i\in \{1,2\}$).

Now note that
\begin{align*}
&\sum_{c_1'c_2'\in \Gamma(c_1c_2)}(\ell(u_1')\ell(u_2'))^{-1} \\
& \le \sum_{c_1'\in L(u_1)} \frac{\deg^*_\sH(c_1')}{\ell(u_1)\min_{v\in N_G(u_1)} \ell(v)} + \sum_{c_2'\in L(u_2)} \frac{\deg^*_\sH(c_2')}{\ell(u_2)\min_{v\in N_G(u_2)} \ell(v)}
 \le 1/4,
\end{align*}
by the assumption on $\deg^*_\sH$. So~\eqref{eqn:glllhyp1} is fulfilled if there is $k>0$ such that
\[
\log k \ge 0.35k \text{ and } k(\ell(u_1)\ell(u_2))^{-1} < 0.5 \text{ for all }u_1, u_2\in V(G).
\]
Noting the lower bound condition on $\ell$, the choice $k= 3$ is enough.
\end{proof}

\section{Local list colouring}
\label{sec:localmolloy}

In this section, we prove Theorem~\ref{thm:localmolloy}. Let us remark that an alternative to the following derivation would be to similarly follow Molloy's original proof and apply Corollary~\ref{cor:finishingblow}. We will sketch a proof of the following stronger form of Theorem~\ref{thm:localmolloy}.

\begin{theorem}\label{thm:localmolloyDP}
Fix $\eps > 0$, let $\Delta$ be sufficiently large, and let $\delta=(192\log \Delta)^{2/\eps}$. Let $G$ be a triangle-free graph of maximum degree $\Delta$ and $\sH = (L,H)$ be a cover of $G$ such that 
\[|L(u)| \ge (1+\eps)\max\Bigg\{\frac{\deg(u)}{\log \deg(u)},\,\frac{\delta}{\log \delta}\Bigg\}\]
for all $u\in V(G)$. Then $G$ is $\sH$-colourable.
\end{theorem}

We will need further notation.
Given a cover $\sH = (L,H)$, the \emph{domain} of an independent set $I$ in $H$ is $\dom(I) = \{u \in V(G)\,:\, I\cap L(u) \neq \varnothing\}$.  Let $G_I = G - \dom(I)$ and let $\sH_I = (L_I, H_I)$ denote the cover of $G_I$ defined by
	\[
	H_I = H - N_H[I] \qquad\text{and}\qquad L_I(u) = L(u) \setminus N_H(I) \text{ for all } u \in V(G_I).
	\]
	Note that, if $I'$ is an $\sH_I$-colouring of $G_I$, then $I \cup I'$ is an $\sH$-colouring of $G$.

For the rest of this section, fix $0<\eps<1$, $\Delta$, $\delta$, $G$, and $\sH$ to satisfy the conditions of Theorem~\ref{thm:localmolloyDP}. 
Write 
\[
k(u)=|L(u)|=(1+\eps)\max\bigg\{\frac{\deg_G(u)}{\log\deg_G(u)},\,\frac{\delta}{\log\delta}\bigg\}\,,
\]
and set $\ell(u) = \max\{\deg_G(u)^{\eps/2}, \delta^{\eps/2}\}$ so that $\ell(u) \ge 192\log \Delta$ for all $u$.

With this notation, and in view of Lemma~\ref{lem:finishingblow}, it suffices to establish the following analogue of Lemma~3.5 in~\cite{Ber19}.

\begin{lemma}\label{lem:3.5}
The graph $H$ contains an independent set $I$ such that
\begin{enumerate}
\item
$|L_I(u)|\ge \ell(u)$ for all $u\in V(G_I)$, and
\item
$\deg^*_{\sH_I}(c) \le 24\log\Delta$ for all $c \in V(H_I)$.
\end{enumerate}
\end{lemma}

In exactly the same way that Lemma~3.5 in~\cite{Ber19} follows from Lemma~3.6 in~\cite{Ber19}, Lemma~\ref{lem:3.5} follows from the following result. We refer the reader to~\cite{Ber19} for further details.

\begin{lemma}\label{lem:3.6}
Fix a vertex $u\in V(G)$ and an independent set $J\subseteq L(\overline{N_G[u]})$. Let $\bI'$ be a uniformly random independent subset of $L_J(N_G(u))$ and let $\bI=J\cup \bI'$. Then
\begin{enumerate}
\item
$\Pr(|L_\bI(u)|< \ell(u)) \le \Delta^{-3}/8$, and
\item
$\Pr\left(\exists c\in L_\bI(u) \, : \, \deg^*_{\sH_\bI}(c) > 24\log\Delta\right) \le \Delta^{-3}/8$.
\end{enumerate}
\end{lemma}

\begin{proof}[Proof sketch]
Since the proof is nearly the same as the proof of Lemma~3.6 in~\cite{Ber19}, we only highlight the essential differences.

The two proofs are completely identical until the application of Jensen's Inequality (``by the convexity\dots''), where we instead get
\begin{align*}
\EE\big|L_\bI(u)\big|
  &\ge k(u)\exp\left(-\frac{\deg_G(u)}{k(u)}\right)
\\&= (1+\eps)\max\bigg\{\frac{\deg_G(u)^{1-1/(1+\eps)}}{\log \deg_G(u)},\, \frac{\delta^{1-1/(1+\eps)}}{\log \delta}\bigg\}
\\&> 2\max\big\{\deg_G(u)^{\eps/2},\, \delta^{\eps/2}\big\} = 2\ell(u)\,,
\end{align*}
where the final inequality holds for $\Delta$ (and hence $\delta$) large enough in terms of $\eps$, because by convexity $1-1/(1+\eps) > \eps/2$ for $0<\eps<1$. The application of a Chernoff Bound for negatively correlated random variables applies in the same way as in Bernshteyn's proof to yield that
\[
\Pr(|L_\bI(u)|< \ell(u)) \le \exp(-\ell(u)/4) \le \Delta^{-48},
\]
which is at most $\Delta^{-3}/8$ for $\Delta \ge 2$.

For the second part of the proof, we instead for all $c \in L(u)$ define
\[
p_c = \Pr\left(c\in L_\bI(u) \text{ and } \deg^*_{\sH_\bI}(c) > 24\log\Delta\right)
\]
and it will suffice to show $p_c \le \Delta^{-4}$. The argument is the same to show that for $\Delta$ large enough in terms of $\eps$,
\[
\EE \deg^*_{\sH_\bI}(c) \le 4\log \Delta\,,
\]
and a similar second application of a Chernoff Bound then yields
\begin{align*}
p_c
& \le \Pr\left(\deg^*_{\sH_\bI}(c) > 24\log\Delta\right)\\
& \le \Pr(\deg^*_{\sH_\bI}(c) > \EE\deg^*_{\sH_\bI}(c) + 20\log\Delta)\\
& \le \Delta^{-20/3} \le \Delta^{-4},
\end{align*}
as required.
\end{proof}

\section{A necessary minimum degree condition for bipartite graphs}
\label{sec:necessary}

In Theorem~\ref{thm:localmolloy} the condition is only truly local when the graph is of minimum degree $\delta=(192\log\Delta)^{2/\eps}$, which grows with the maximum degree $\Delta$. 
The result is made strictly stronger by reducing $\delta$. 
In this section we show that even for bipartite graphs the conclusion of Theorem~\ref{thm:localmolloy} requires some $\omega(1)$ bound on $\delta$ as $\Delta\to\infty$. We state and prove the result specifically with $\deg(u)/\log\deg(u)$ as the target local list size per vertex $u$. The reader can check that any sublinear and superlogarithmic function will do, but with a different tower of exponentials.

\begin{proposition}\label{prop:necessary}
For any $\delta$, there is a bipartite graph of minimum degree $\delta$ and maximum degree $\exp^{\delta-1}(\delta)$ (so a tower of exponentials of height $\delta-1$) that is not $L$-colourable for some list assignment $L : V(G)\to 2^{\EZ^+}$ satisfying \[|L(u)| \ge \frac{\deg(u)}{\log \deg(u)}\] for all $u\in V(G)$.
\end{proposition}

\begin{proof}
The construction is a recursion, iterated $\delta-1$ times.

For the basis of the recursion, let $G_0$ be the star $K_{1,\delta}$ of degree $\delta$. We write $A_0$ as the set containing the centre $v_0$ of the star and $B_0$ as the set of all non-central vertices. Note that, with the assignment $L_0$ that assigns the list $\{1_0,\dots,\delta_0\}$ to the centre and lists $\{i_0\}$, $i\in[\delta]$, to the non-central vertices, $G_0$ is not $L_0$-colourable.

We recursively establish the following properties for $G_i$, $A_i$, $B_i$, $L_i$, where $0\le i \le \delta-1$:
\begin{enumerate}
\item $G_i$ is bipartite with partite sets $A_i$ and $B_i$;
\item $A_i$ has all vertices of degree at least $\delta$ and at most $\exp^i(\delta)$, with some vertex $v_i$ attaining the maximum $\exp^i(\delta)$;
\item $B_i$ has $\exp^i(\delta)$ vertices of degree $i+1$;
\item $|L_i(a)| \ge \deg(a)/\log \deg(a)$ for all $a\in A_i$ and $|L_i(b)| \ge \deg(b)$ for all $b\in B_i$; and
\item $G_i$ is not $L_i$-colourable.
\end{enumerate}
These properties are clearly satisfied for $i=0$.

From step $i$ to step $i+1$, we form $G_{i+1}$ by taking $\exp(\exp^i(\delta))/\exp^i(\delta)$ copies of $G_i$ and adding a vertex $v_{i+1}$ universal to all of the $B_i$-vertices. Let $A_{i+1}$ be $v_{i+1}$ together with all $A_i$-vertices, and $B_{i+1}$ be all of the $B_i$-vertices. Label each copy of $G_i$ with $j$ from $1$ to $\exp(\exp^i(\delta))/\exp^i(\delta)$. We set $L_{i+1}(v_{i+1}) = \{1_{i+1},\dots,\exp(\exp^i(\delta))/\exp^i(\delta)_{i+1}\}$ and add colour $j_{i+1}$ to $L_i(b)$ to form $L_{i+1}(b)$ for every $B_i$-vertex $b$ in the $j$th copy of $G_i$. It is routine to check then that $G_{i+1}$, $A_{i+1}$, $B_{i+1}$, $L_{i+1}$ satisfy the promised properties.

The proposition follows by taking $G_{\delta-1}$.
\end{proof}

As a final remark on minimum degree or minimum list size conditions, we note that our proof of Theorem~\ref{thm:localmolloy} can be adapted to reduce $\delta=(192\log\Delta)^{2/\eps}$ as a function of $\Delta$ by increasing the leading constant `1' in the list size condition. 
Indeed, this removes the dependence on $\eps$ and brings the result much closer to the triangle-free case of the significantly more general local colouring result of Bonamy \emph{et al.}~\cite{BKNP18+}, which has a minimum degree condition of $(\log\Delta)^2$. 
Here, as we focus on triangle-free graphs we prefer to aim for the best possible constant at the expense of the cutoff value $\delta$.

\bibliography{frac}
\bibliographystyle{habbrv}

\end{document}